\documentclass[openright,a4paper,american,11pt]{amsart}
\usepackage[latin1]{inputenc}

\usepackage[dvips]{graphicx}
\usepackage{amsmath}
\usepackage{amssymb}
\usepackage{hyperref}
\usepackage{amsfonts}
\usepackage[usenames]{color} 
\usepackage{colortbl}
\usepackage{tikz}
\usetikzlibrary{decorations.markings}
\input xy
\xyoption{all}
\def\tp{t^{\frac{1}{2}}}
\def\tm{t^{-\frac{1}{2}}}
\def\s{\sigma}
\def\td{(\tm - \tp)}
\def\R{\mathbb R}
\def\Z{\mathbb Z}
\def\ZZ{\mathbb Z[\tp,\tm]}
\def\Symm{\mathrm{Symm}}
\newcommand{\lin}[6]{
\draw[decoration={markings, mark=at position 0.5 with {\arrow{#6}}},postaction={decorate}] (#1,#2) .. controls (#1+#5,#2+#4) and (#3-#5,#2+#4)..(#3,#2);
}

\newcommand{\linn}[5]{
\draw(#1,#2) .. controls (#1+#5,#2+#4) and (#3-#5,#2+#4)..(#3,#2);
}

\newtheorem{theo}{Theorem}
\newtheorem{thm}{Theorem}[section]
         {\theoremstyle{definition}
\newtheorem{defi}[thm]{Definition}}
        \newtheorem{prop}[thm]{Proposition}
\newtheorem{lemma}[thm]{Lemma}
\newtheorem{cor}[thm]{Corollary}
\newtheorem{rem}[thm]{Remark}
         {\theoremstyle{definition}
\newtheorem{ex}[thm]{Example}}

\hypersetup{pdfstartpage=1, pdftitle="Alexander polynomial as an
  intersection of two cycles in a symmetric power", pdfauthor="Nikita
  Kalinin", pdfstartview={FitH}, pagebackref=true, pdfkeywords={Alexander polynomial, local system, symmetric power} }
\begin{document}
\keywords{braid group action, the Alexander polynomial.}

\title{The Alexander polynomial as an intersection of two cycles in a
  symmetric power}

\date{\today}

\author[N. Kalinin]{Nikita Kalinin}

\email{Nikita.Kalinin\{dog\}unige.ch, nikaanspb\{dog\}gmail.com}

\begin{abstract}
We consider a braid $\beta$ which acts on a punctured plane. 
Then we construct a local system on this plane and find a homology cycle $D$ 
in its symmetric power, such that $D\cdot \beta(D)$ coincides with 
the Alexander polynomial of the plait closure of $\beta$. 
\end{abstract}
\maketitle

\section{Introduction}

\subsection{Results obtained here}
We construct the Alexander polynomial for a knot $K\subset S^3$ presented as the plait closure of a colored braid. For that we use a local system on the plane imitating the action of $\mathbb Z[t, t^{-1}]$ on
the universal abelian covering of $S^3\setminus K$. Our construction is similar to a Floer homology construction: we intersect two
 manifolds in a symmetric power of a surface. We also obtain the presentation matrix of the Alexander module.
In the article \cite{Bigelow} (based on \cite{Lawrence}) the Jones polynomial of a knot is presented as
the intersection of two homology classes in a covering over a symmetric
power of the punctured disk; this article is based on
similar ideas. Lawrence's approach might lead to a construction of
higher Jones polynomials, which are currently unknown; that was one of the
main motivations for this work. Section \ref{sec_local} contains the definition of our local system, Section \ref{sec_braid} contains the preliminaries about colored braids, Section \ref{sec_defi} contains the definition of the invariant, statements of the theorems and their interpretation for the symmetric powers. The explicit formulae for the braid group action can be found in Section \ref{sec_action}, whereas Section \ref{sec_aux} consists of some auxiliary technical statements. Sections \ref{sec_proof},\ref{sec_skein} contain the proof and Section \ref{sec_rem} is devoted to remarks.

\subsection{Acknowledgements} This work appeared as my (master) graduation thesis in Saint Petersburg State University, 2010. 
I thank my adviser Oleg Yanovich Viro who gave me the problem discussed here. I thank
A. Akopyan, M. Karev, S. Podkorytov, L. Positselsky and participants of the topology seminar for discussions, suggestions and simplifications. Recently S. Bigelow told me
that the presented construction is not new and it is definitely known for experts. But, since it is hard
to find it written explicitly elsewhere, I translated the article into
English. Soon after I uploaded the draft to arxiv, Vincent Florens communicated to me that
the same construction of the Alexander polynomial is contained in
their paper \cite{vincent}(chapter 3), whose first version was written in
2008. Hence, my paper may serve as an example of concrete
computations. Research is supported in part by the grant 140666 of the Swiss National Science
Foundation as well as by the National Center of Competence in Research
SwissMAP of the Swiss National Science Foundation.

\subsection{Short history of the Alexander polynomial}

J. Alexander, using a diagram of a knot $K$, constructed the polynomial 
$\varDelta_K(t)$. This polynomial is invariant (modulo multiplications by 
$\pm t^k$) under Reidemeister moves, and, therefore, $\varDelta_K(t)$ depends only on the isotopy class
of $K$ (\cite{Alexander}). Alexander's construction works as follows:
to each connected component of the complement of the knot projection we associate
a variable, and to each crossing in the projection we associate an
equation in variables corresponding to the components touching this crossing. These equations are linear and
contain a formal parameter $t$. This gives us an $n\times (n+2)$ matrix $M$
consisting of the coefficients in all these equations; here $n$ is the number of crossings on the knot diagram. 
It happens that if we remove any two columns from $M$, then the
determinant $\varDelta_K(t)$ of the rest is an invariant (modulo multiplication by $\pm t^k$) with respect to the Reidemeister moves.
Afterward, this invariant was called the Alexander polynomial.
 
Each knot $K\subset S^3$ bounds a compact oriented surface $S$,
which is called a {\it Seifert surface} of $K$. The genus $g(K)$ of $K$ is, by definition, the minimal
possible genus for a Seibert surface spanning $K$. Then,
H. Seifert (\cite{Seifert}) found that $2g(K)$ is bounded from below
by the degree of
the Alexander polynomial of $K$, this required a new way to calculate the latter. Seifert considered the infinite
cyclic covering over $S^3\setminus K$; we obtain this covering if we cut
$S^3$ along a Seibert surface of $K$ and then glue together a countable number of such pieces along
their boundaries. The cohomology group $H^1$ of the obtained space is a module over
$\mathbb Z[t,t^{-1}]$, and the determinant of its representation matrix
is the Alexander polynomial.

An equivalent way to obtain this determinant is to take loops presenting a basis of the first homology group
of the Seibert surface, slightly push them from this surface to one side and then
compute linking numbers between the loops on the surface and
the moved loops. The resulting matrix $A$ of linking numbers is called {\it a Seifert matrix}.
Now we can compute the Alexander polynomial as $\varDelta_K(t) =
det(tA - A^{T})$ where $A^T$ is the transpose of $A$.

After the work of Alexander, R. Fox considered a presentation
of the knot group $\pi_1(S^3\setminus K)$, and introduced a non-commutative differential calculus (\cite{Fox61}),
which also allows us to compute $\varDelta_K(t)$. Detailed exposition
of this approach and information about higher Alexander polynomials (a.k.a. elementary ideals of the Alexander matrix) can be found
in the book \cite{CrowFox}.

J. Conway ruled out problems with sign in the polynomial and defined it axiomatically
via skein-relations \cite{Conway}.

There are a number of constructions of the Alexander polynomial via state sums derived
from physical models. The first such a construction appeared in the article \cite{Kauffman83}. 
A survey of this topic and other connections with physics can be found in \cite{Kauffman}.

In \cite{OzsvathSzabo}, \cite{Rasmussen} the Alexander polynomial is presented
as the Euler characteristic of a complex, whose homology are isotopy invariants of
the considered knot $K$, therefore Floer homology theory is a {\it
  categorification} of the Alexander polynomial. The reader can find more details in the survey \cite{Khovanov}.

\subsection{Some properties of the Alexander polynomial} The Alexander
polynomial has many beautiful properties. For example, if a knot $K$
is slice, i.e. being a knot in $S^3$ it bounds an embedded locally flat disk in
$D^4, \partial D^4=S^3$, then  
$\varDelta_K(t) = f(t)f(t^{-1})$ where
$f\in \Z[t]$(\cite{FoxMilnor}), furthermore, the degree of Alexander polynomial is no more than $2g(K)$ (\cite{Seifert}, see also \cite{Lickorish}).

Recently, an extraordinary connection between the Alexander polynomial 
and Seiberg-Witten invariants of the smooth four-dimensional manifolds has been found
(\cite{FintushelStern}, Seiberg-Witten invariants can recognize different smooth
structures on the same topological manifold).
It happens that if we cut out a neighborhood of a torus from a four-dimensional manifold $M$
and glue $S^1\times(S^3\setminus K)$ into the obtained free space, then, with some additional
assumptions, the obtained manifold $M_K$ will be homeomorphic to $M$,
and the Seiberg-Witten invariant will change by multiplication by $\varDelta_K(t)$.

\section{Local system of coefficients}\label{sec_local}

Let $D\subset \R^2$ be the unit disk on the plane, and let
$p_1,p_2,\ldots,p_{2n}\in D$ be a collection of {\it marked points} which lie
on the $x$-axis in the index increasing order: $-1<p_1<p_2<\ldots<p_{2n}<1$. Let us color points with odd indices
in black, and the points with even indices in white. Denote $D' = \R^2\setminus{\bigcup_{i=1}^n \{p_i\}}$, we shall call $D'$ a {\it punctured disk}, though $D'$ is only homotopically equivalent to it. 

In the following text we treat $t$ as a formal variable. We
consider an abelian local system $\Theta$ on $D'$ with the fiber $\mathbb
Z[\tp,\tm]$ : a small counterclockwise
rotation around a point $p_i$ corresponds to multiplication by $t^{(-1)^{i+1}}$
in the fiber. We choose the following trivialization of
  $\Theta$ near the $x$-axis:

\begin{defi}
Let $\Theta$ be trivial on the $x$-axis. Then, consider points
$(x,0),(y,0)$ such that $x<p_i<y$
and $[x,y]$ contains no marked points except $p_i$. Let the
path from $(x,0)$ to $(y,0)$ by a semicircle in the bottom
(resp. upper) half-plane gives multiplication by $t^{\frac{(-1)^{i+1}}{2}}$
(resp. by $t^{\frac{(-1)^i}{2}}$) in the fiber of $\Theta$. 
\end{defi}
The semicircles considered above are called {\it basic semicircles}. Each loop in $D'$ is homotopy equivalent to a sequence of basic semicircles.

So, the local system is fixed near the $x$-axis and we will not need its
concretization on the rest of the plane. Notice that an extension of
$\Theta$ from the $x$-axis to the plane is homotopically unique because the
bottom and the upper half-planes are simply-connected.

\begin{rem}
The expression ``total space of $\Theta$'' 
means the total space of the local system $\Theta$. The fibers of $\Theta$ are equipped with the discrete topology, therefore
the total space of $\Theta$ is a covering. 
\end{rem}

We consider two types of circles $s'_i,d'_i$ in $D'$. Namely,
 for $i=1,...,n$ the diameter
of $s'_i$ is an interval,
slightly bigger than the interval $[p_{2i-1},p_{2i}]$; therefore $s'_i$ contains no
marked points except $p_{2i-1},p_{2i}$. Similarly, for $i=1,...,n-1$  the diameter
of $d'_i$ is an interval slightly bigger than $[p_{2i},p_{2i+1}]$; therefore $d'_i$
contains no marked points except $p_{2i},p_{2i+1}$. We orient all these circles counterclockwise.

For each  $i=1,\dots, n$ we can lift $s'_i$ (resp. $d'_i$ for $i=1,\dots,n-1$) into the total space of $\Theta$ in
such a way that the points of these liftings over the $x$-axis have coordinate $1$ (resp. $\tp$) in the fiber of the local system. 
Denote these lifted circles by $s_i$ and $d_i$ correspondingly.

We shall add one more cycle $d_n$ to our collections $s_i(i=1,\dots,n),d_i(i=1,\dots,n-1)$. 
Let us take the semicircle with diameter 
$[p_1,p_{2n}]$ and take a loop $d_n'$ which is the boundary of a small
neighborhood of this semicircle, see Fig. \ref{my}. We define its
canonical lifting $d_n$ to the total space of $\Theta$ such that $d_n$ has the
coordinates $\tp$ in the fibers over the second
and third points of the intersection of $d_n'$ with the $x$-axis. 

\begin{figure}[h]
\begin{center}
\begin{tikzpicture}[scale=1]
\newcommand{\ba}{5};
\newcommand{\ra}{0.1};
\draw[dotted] (-4,0)--(3,0);
\draw[fill] (-3,0) circle (\ra) ;
\draw (-2,0) circle (\ra) ;
\draw[fill] (-1,0) circle (\ra) ;
\draw (0,0) circle (\ra) ;
\draw[fill] (1,0) circle (\ra) ;
\draw (2,0) circle (\ra) ;
\lin{-3.2}{0}{-1.8}{0.5}{0.3}{<}
\lin{-1.2}{0}{0.2}{0.5}{0.3}{<}
\lin{0.8}{0}{2.2}{0.5}{0.3}{<}

\lin{-3.2}{0}{-1.8}{-0.5}{0.3}{>}
\lin{-1.2}{0}{0.2}{-0.5}{0.3}{>}
\lin{0.8}{0}{2.2}{-0.5}{0.3}{>}

\lin{-3.5}{0}{-2.5}{-0.5}{0.3}{>}
\lin{1.5}{0}{2.5}{-0.5}{0.3}{>}
\lin{-2.5}{0}{1.5}{1.5}{0.3}{>}
\lin{-3.5}{0}{2.5}{2.5}{0.3}{<}
\draw (-0.5,1.5) node {$d_3$};
\draw (-2.5,-0.7) node {$s_1$};
\draw (-0.5,-0.7) node {$s_2$};
\draw (1.5,-0.7) node {$s_3$};
\draw (-2.5,0.8) node {$\tp$};
\draw (1.5,0.8) node {$\tp$};
\draw (-3.5,0.8) node {$1$};

\end{tikzpicture}

\caption{Added $d_n$, the numbers above the $x$-axis are the
  coordinates of the intersections of $d_n$ with the $x$-axis.}\label{my}
\end{center}
\end{figure}
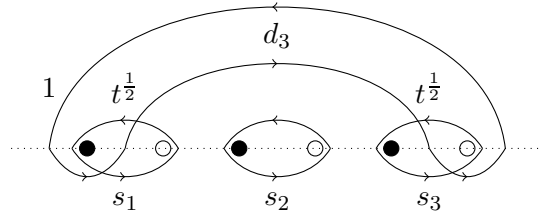

If $p_1,\dots, p_{2n}$ are arranged uniformly on a circle rather than
linearly, so that $p_{2n}$ is adjacent to $p_1$, then the definitions
of  all $d_i(i=1,...,n)$ will be alike. Also, we
abuse notation by writing $d_0=d_n, s_0=s_n, d_1=d_{n+1}$, etc.

We denote by $[x]$ the homology class of a chain $x$. However, we often omit brackets and write $x$ instead of $[x]$ where it simplifies the notation and can not lead to misunderstanding. 
\begin{prop}
The first homology group (with coefficients
in local system $\Theta$) $H_1(D';\Theta)$ of the punctured disk is freely
generated (as a module over $\mathbb
Z[\tp,\tm]$) by the set $\{[s_i]|i=1,...,n\}\cup\{
[d_i]|i=1,...,n-1\}$.
\end{prop}
\begin{proof} The vanishing of $H_{k>1}(D';\Theta)$ follows from the fact that $D'$ 
is homotopically equivalent to a bouquet of $2n$ circles. We denote by
$C_1,C_2,\dots C_{2n}$ the $1$-cells of the 
bouquet, and by $p$ its $0$-cell. Then we orient all $C_i$ counterclockwise and lift them in the total
space of $\Theta$ such that all liftings $C_i'$ will start from a point $p$ with coordinate $1$ in the fiber over $p$.
In this construction we have $s_i=C_{2i-1}'+\tp\cdot C_{2i}'$ and $d_i= C_{2i}'+\tm\cdot C_{2i+1}'$.
Now we consider any exact $1$-chain $a=\sum a_iC_i$. Exactness means that $p\cdot(\sum a_{2i-1}(\tp-1)+\sum a_{2i}(\tm-1))=0$.
Now we subtract the exact chain $a_1s_1$ from $a$ and this
kills $C_1$ there. Then we subtract $d_1,s_2,d_2,s_3,\dots$ with some
coefficients, each time killing the coefficient before $C_i$ for the
next $i$.  A chain
$kC_{2n}$ is not exact for $k\ne 0$, therefore we proved that
all exact chains are generated by $s_i,d_i$.
\end{proof}

\begin{rem}
One can show that ${H_0(D';\Theta) = \ZZ/{(\tp-1)}} = \mathbb Z$ and $\sum_{i=1}^n [d_i] = \sum_{i=1}^n [s_i]$ in $H_1(D';\Theta)$.
\end{rem}

Let $pr\colon \Theta \to D'$ be the natural projection. We consider
a loop $l$ in the total space of the local system $\Theta$. It follows from the homotopy lifting property that,
given a point $x\in l$ in some fiber and $pr(l)$ in the plane, we can uniquely determine $l$.

{\bf How to understand figures.} 
It is convenient to depict (up to homotopy) a loop $l$ in the total space $\Theta$ using the following convention:
we take the projection of $l$ in $D'$ and deform it homotopically into {\it basic} semicircles, such that one of them contains the point $(-1,0)$. Therefore we already depicted $pr(l)$; its place in the total space of $\Theta$ 
 is encoded by coordinates of $l$ in the fibers over intersection of $pr(l)$ with the $x$-axis, see Fig.\ref{fig1}. Also, by homotopy lifting property, it is enough to know the coordinate of $l$ over the point $(-1,0)$, cf. Fig.~\ref{fig2}. Consider the semicircles from the leftmost point of each of $s_i',d_i'$ to the point $(-1,0)$. We can lift these semicircles altogether with $s_i',d_i'$. Note, that we chose liftings $s_i,d_i$ such that these lifted semicircles have the coordinate $1$ in the fiber over $(-1,0)$, see Fig.\ref{fig2}.

\section{The colored braid group and $H_1(D',\Theta)$}\label{sec_braid}

Take an arbitrary oriented link $L\subset \R^3$. Choose a direction ($z$-axis) and deform $L$ isotopically by pushing down all local
minima of $z$-coordinate and stretching up all local maxima of $z$-coordinate. This gives us a presentation of $L$ as the plait
closure (Fig.\ref{fig}) of a braid with $2n$ strings where $n$ is the number of local maxima. Twisting, if necessary, the neighbors of maxima and minima,  we may assume that the strings oriented upwards
connect the points with odd indices  and the string oriented downwards connect the points with even indices.  

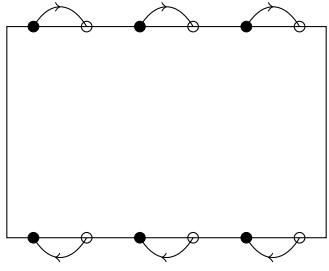
\begin{figure}[h]
\begin{center}
\begin{tikzpicture}[scale=0.7]
\newcommand{\ba}{4};

\lin{1}{\ba}{2}{0.5}{0.3}{>}
\lin{1}{0}{2}{-0.5}{0.3}{<}
\lin{3}{\ba}{4}{0.5}{0.3}{>}
\lin{3}{0}{4}{-0.5}{0.3}{<}
\lin{5}{\ba}{6}{0.5}{0.3}{>}
\lin{5}{0}{6}{-0.5}{0.3}{<}
\draw (0.5,0)--(6.5,0)--(6.5,\ba)--(0.5,\ba)--cycle;
\newcommand{\ra}{0.1};
\draw[fill] (1,\ba) circle (\ra) ;
\draw (2,\ba) circle (\ra) ;
\draw[fill] (3,\ba) circle (\ra) ;
\draw (4,\ba) circle (\ra) ;
\draw[fill] (5,\ba) circle (\ra) ;
\draw (6,\ba) circle (\ra) ;
\draw[fill] (1,0) circle (\ra) ;
\draw (2,0) circle (\ra) ;
\draw[fill] (3,0) circle (\ra) ;
\draw (4,0) circle (\ra) ;
\draw[fill] (5,0) circle (\ra) ;
\draw (6,0) circle (\ra) ;

\end{tikzpicture}
\caption{The plait closure of a braid colored in two colors}\label{fig}
\end{center}
\end{figure}
Let $B_{n,n}$ be the colored braid group with $2n$ strings, where
$n$ strings are colored in black (and we call these strings {\it
  odd}) and the other $n$ strings are colored in white (we
call these strings {\it even}). In this article we consider only braids with such a coloring.
Let
us read braids from bottom to up, odd strings are responsible for the
movement of odd (black) points $\{p_{2i+1}\}_{i=1}^n \subset D$ and even strings are responsible for the movement of even (white) points $\{p_{2i}\}_{i=1}^n\subset D$.  In multiplication, for example, $\beta\eta$, we add $\beta$ below $\eta$, see Fig.~\ref{fig:fig6} for an example of conventions we use.

\begin{figure}[hb]
\begin{center}
\begin{tikzpicture}[scale=1]

\newcommand{\ra}{0.1};
\draw[dotted] (-4,0)--(3,0);
\draw (0,0) circle (\ra) ;
\draw[fill] (-1,0) circle (\ra) ;
\draw (-2,0) circle (\ra) ;
\draw[fill] (-3,0) circle (\ra) ;
\draw (2,0) circle (\ra) ;
\draw[fill] (1,0) circle (\ra) ;
\draw (-3,0) node[below]{$p_1$};
\draw (-2,0) node[below]{$p_2$};
\draw (1,0) node[below]{$p_5$};
\draw (2,0) node[below]{$p_6$};
\draw (-2.5,1) node {$s_1$}; 
\draw (-0.5,1) node {$s_2$};
\draw[->] (-1.8,0) arc (0:90:0.7);
\draw[->] (-1.8,0) arc (0:180:0.7);
\draw[->] (-3.2,0) arc (180:360:0.7);
\lin{-0.5}{0}{0.5}{0.3}{0.3}{<}
\lin{-0.5}{0}{0.5}{-0.3}{0.3}{>}
\lin{-1.5}{0}{-0.5}{0.3}{0.3}{<}
\lin{-1.5}{0}{-0.5}{-0.3}{0.3}{>}

\draw (-1.5,-0.1) node[below]{$1$};
\draw (0.5,-0.1) node[below]{$1$};
\draw (-0.5,-0.1) node[below]{$a^{-1}$};
\draw (-0.5,0.1) node[above]{$a$};
\end{tikzpicture}

\caption{Filled points are black, empty points are white, semicircles
  are oriented counterclockwise,  $s_1$ is presented as in the
  definition, $s_2$ is presented via basic semicircles.
The number $a=\tm$ on top is the coordinate of $s_2$ in the fiber over the intersections of the top half of $s_2$ with the $x$-axis. The symmetric rule is applied for the bottom numbers. The ratio between the nontrivial bottom and 
top numbers is $\frac{a^{-1}}{a} = t$, as it should be for the end points
of a counter-clockwise oriented loop around a marked point with odd index.}
\label{fig1}
\end{center}
\end{figure}
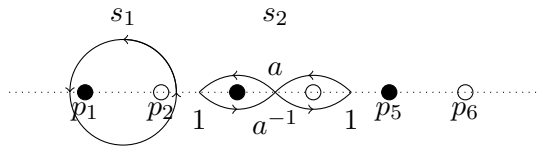

\subsection{The action of $B_{n,n}$ on the total space of $\Theta$}
\label{action}
An action of a group $G$ on a topological space $X$ is given by a
homomorphism $\phi\colon G \to Homeo(X)$ from $G$ to the automorphism group of the space $X$.

The group $B_{n,n}$ acts on $D$, mapping marked points to
marked points preserving colors. Namely, we have a homomorphism 
$\phi\colon B_{n,n} \to \pi_0(Homeo(D'))$ where we put by definition
$${Homeo(D') =
Homeo(D,\partial D, \bigcup_{i=1}^n\{p_{2i+1}\},\bigcup_{i=1}^n\{p_{2i}\})}$$ to be the group of autohomeomorphisms $f$
of $D$, such that $f|_{\partial D}=id$, and for each $i=1,\dots, n$ we have $f(p_i)=p_j$ where $i\equiv j \pmod 2$. 

Let $Homeo(D',\Theta)$ be the group of the autohomeomorphisms of the total space of $\Theta$, 
which preserve the fibers of $\Theta$ and are identical outside of $D$.
 
\begin{prop} The natural projection $Homeo(D',\Theta)\to Homeo(D')$ is an isomorphism.\qed
\end{prop}

We have an isomorphism
$\pi_0(Homeo(D'))\to\pi_0(Homeo(D',\Theta))$.
Hence the homomorphism $B_{n,n} \to
\pi_0(Homeo(D',\Theta))$ gives us the action of $B_{n,n}$ on
$H_*(D';\Theta)$.

Now we show how we calculate this action.
Let $\beta\in B_{n,n}$ be a colored braid, $l$ be a loop in
the total space of $\Theta$. We know the action of $\beta$ on $D'$, therefore $pr(\beta l) =
\beta \cdot pr(l)$. To completely determine $\beta l$ we need only one point in $\beta l$. 
To do this we  {\it anchor} $l$, that means that 
in the homotopy class of $l$ we choose a loop which passes through $(-1,0)$. 
Now we know that $\beta ((-1,0))=(-1,0)$ therefore the coordinates of $l$ and
$\beta(l)$ in the fiber over $(-1,0)$ are the same (see
Fig.\ref{fig2}), then we use homotopy lifting property.

\begin{figure}[h]
\begin{center}
\begin{tikzpicture}[scale=1]
\newcommand{\ra}{0.1};
\draw[dotted] (-5,0)--(1,0);
\draw (0,0) circle (\ra) ;
\draw[fill] (-1,0) circle (\ra) ;
\draw (-2,0) circle (\ra) ;
\draw[fill] (-3,0) circle (\ra) ;
\draw (-4,0) node[below]{$(-1,0)$};
\draw (-4,0.1) node[above]{$1$};
\draw (-1.2,0.1) node[above]{$1$};
\draw (0.2,0.1) node[above]{$1$};
\draw (-3,0) node[below]{$p_1$};
\draw (-2,0) node[below]{$p_2$};
\draw (-1,0) node[below]{$p_3$};
\draw (0,0) node[below]{$p_4$};
\draw (-0.5,0.7) node {$s_2$}; 
\lin{-4}{0}{-1.2}{0.5}{0.5}{>}
\lin{-1.2}{0}{0.2}{0.5}{0.3}{<}
\lin{-1.2}{0}{0.2}{-0.5}{0.3}{>}

\draw (-2.6,0.6) node{$s'$};
\end{tikzpicture}

\caption{A homological cycle $[s_2]$ is presented as the loop parametrized by $s's_2s'^{-1}$, where $s'$
  is the path from the point $(-1,0)\in D'$ to a point on $s_2$. Numbers above $x$-axes are the
  coordinates of the loop in the fibers over the $x$-axes.}\label{fig2}
\end{center}
\end{figure}
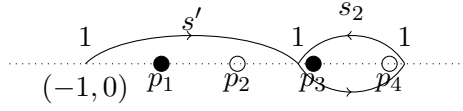

\subsection{Intersections with coefficients in the local system}
\label{intersection}

We can think of each element of $H_*(D',\Theta)$ as
a chain living in the total space of $\Theta$. Suppose that we have a cycle $a$ in $H_1(D',\Theta)$. For $i=1,\dots,n$ denote by $t^{\frac{k}{2}}r_i'$ the lifting of
the interval $r_i'=(p_{2i-1},p_{2i})$ in the total space of $\Theta$ such that the lifted points have coordinate $t^{\frac{k}{2}}$
in the fibers. Now define $a\cdot r_i' = \sum_{k\in \mathbb Z} <a
\cdot t^{\frac{k}{2}}r_i'>t^{\frac{k}{2}}$ where $<\cdot,\cdot>$ is
the usual intersection product of two homology cycles in a topological space.

\begin{defi} Denote by $r_i\in
H^1(D';\Theta^*), i=1,\dots,n$ the cohomological class, such that
$$r_i\times [s_j] = 0\ (j=1,\dots,n), r_i\times [d_{i-1}] = -1, r_i\times [d_i]
= 1, r_i\times [d_j] = 0\ (j\neq i-1,i).$$
\end{defi}
We denote by $\times$  the natural pairing between homology and
cohomology groups.

\begin{rem}
Using the duality between $H^1(D',\Theta^*)$ and $H_1(D',\bigcup\{p_i\};\Theta)$ we
see that the class $r_i$ is dual to the interval $r_i'=(p_{2i-1},p_{2i})$.
\end{rem}

\begin{prop} If $\beta s_i'$ is in a general position with respect to the $x$-axis, then $r_k\times
\beta [s_i]$ is the sum of fiber coordinates of $\beta s_i$ over points of set-theoretical
intersection of $\beta s_i'$ and $r'_i$.\qed
\end{prop}
%
%
%

\section{Definition of the invariant via external power}\label{sec_defi}

For $\beta\in B_{n,n}$ let $\beta_1\in B_n$ be the 
braid consisting of the odd strings of  $\beta$, and
let $\beta_2\in B_n$ consist of the even strings of $\beta$. Consider the group homomorphism 
$\phi\colon B_n\to \mathbb Z$, which sends all standard generators $\s_i$ to $1$. 

\begin{defi}
\label{def_e} Let $e(\beta) =
(\tp)^{\phi(\beta_1)}\cdot(-\tm)^{\phi(\beta_2)}\cdot (-1)^n.$
\end{defi}

Consider the $(n-1)$-th external power of cohomologies of the punctured
disc. For each $\beta\in B_{n,n}$ we define $U_\beta\in\ZZ$ by
the following formula: $$U_{\beta} = ((n-1)!)^2(r_1\wedge r_2\wedge \dots
\wedge r_{n-1})\times ([\beta s_1]\wedge \dots \wedge [\beta s_{n-1}])
=$$

$$=\det(r_i([\beta s_j)]) = \sum_{\sigma\in
S_{n-1}} \varepsilon(\sigma)\prod_{i=1}^{n-1}r_i\times [\beta
s_{\sigma(i)}].$$

Let $L$ be the plait closure (Fig.\ref{fig}) of a braid $\beta\in B_{n,n}$.
\begin{defi}
Define $U_L = e(\beta) \cdot U_{\beta}$.
\end{defi}

The main results of this paper are
\begin{theo}\label{th_main}
The polynomial $U_L$ is well-defined, i.e. it
does not depend on a presentation of $L$ as the plait closure of some braid
$\beta$.
\end{theo}

\begin{theo}\label{th_2}
The polynomial $U_L$ is the Alexander polynomial $\nabla_L$ of
the link $L$ in Conway normalization (we will prove this by skein
relations).
\end{theo}

We recall that $s_i$ 
{\it morally} is a small circle
containing $[p_{2i-1},p_{2i}]$ and $r_i$ {\it morally} is
$(p_{2i-1},p_{2i})$, therefore we can interpret the above external
powers as manifolds in a symmetric space. 

\subsection{Reformulation in terms of $\Symm_{n-1}(D')$}

Let $\Symm_k(M)$ be the space of all unordered tuples of $k$ points
(points may coincide) in a topological space $M$.

$$\Symm_k(M) = M^k/S_k,\ \  M^k = \underbrace{M\times M\times\ldots\times M}_{\text{$k$
copies}}$$ where $S_k$ acts on the product by permutations.

We prove that a local system $\Theta$ on $D'$
with fiber $\ZZ$ canonically lifts to $\Symm_k(D')$. 
Indeed, such a local system on $D'$ is given, up to homotopy, by a homomorphism $Ab(\pi_1(D'))
\to \mathrm{Aut}(\ZZ)$ from the abelianization of the fundamental group of $D'$ to the group of the ring automorphisms of the fiber. The inclusion $D' \to (D',1,\dots,1)\subset D'^k$, induces the map $D'\to
\Symm_k{D'}$, and hence the homomorphism $Ab(\pi_1(D')) \to
Ab(\pi_1(\Symm_k(D')))$, which is an isomorphism in our case. Therefore
we have the canonical map $Ab(\pi_1(\Symm_k(D')))\to Ab(\pi_1(D'))
\to Aut(\ZZ)$. We denote by $\tilde\Theta$  the obtained local system on $\Symm_k(D')$.

Homology with coefficients in the local system
$\Theta$ of the punctured disc $D'$ are concentrated in dimensions $0$ and $1$, therefore
$$H_*(\Symm_{n-1}(D'); \tilde\Theta) = \bigwedge H_1(D'; \Theta).$$
The action of $\beta$ on $D'$ extends to the action on
$H_*(\Symm_{n-1}(D')$ as $\beta(x_1\wedge x_2\wedge\dots)=\beta
x_1\wedge \beta x_2\wedge\dots$.


Thus, we reformulate the definition of $V_\beta$ and construct two elements $$R' \in
H^{n-1}(\Symm(D');\tilde\Theta^*),R'=r_1\wedge r_2\wedge \dots
\wedge r_{n-1},$$ 
$$S\in
H_{n-1}(\Symm_{n-1}(D'); \tilde\Theta), S=[s_1]\wedge [s_2]\wedge \dots
\wedge [s_{n-1}],$$ and $V_{\beta}$ coincides with the value of $R'$ on $\beta S$.

Now replace $R'$, by duality, with a relative homology class $R\in H_{n-1}(\Symm_{n-1}(D');
\tilde\Theta)$, the closure of $R$ in $\Symm_{n-1}(D)$ is the symmetric product of the intervals 
 $[p_{2i-1},p_{2i}], i =1,...,n-1$. This presents $U_\beta$
as the intersection (with respect to $\tilde\Theta$) of these classes, i.e. $U_\beta=<R, \beta S>$.

Now we give a geometric definition of $U_\beta$:

\begin{defi} Consider two submanifolds in $\Symm_{n-1}(D')$: the
first, $S_{n-1}$, consists of  all sets of $n-1$ points in  $D'$, one
point in each circle $s_1,\ldots,s_{n-1}$, and the second submanifold
$R_{n-1}$ consists of all sets of $n-1$ points in $D'$, one point from
each interval
${(p_{1},p_{2}),\ldots,(p_{2n-3},p_{2n-2})}$.
\end{defi}

We recall how to define the intersection with respect to a local
system. The manifold $R_{n-1}$ is a cube, and we lift it to the total space such that all its points
have coordinate $1$ in fibers. We say that
$t^l S_{n-1}$ is the image of $S_{n-1}$ by the monodromy action which
corresponds to the multiplication by $t^{l}$. 
Let us intersect $R_{n-1}, t^l S_{n-1}$ as usual manifolds, this
intersection consists of a number of points, that is, an
integer number.

\begin{defi} Define the intersection product as $<R_{n-1},S_{n-1}> = \sum_{k\in \mathbb Z}
<R_{n-1},t^{\frac{k}{2}}S_{n-1}>t^{-\frac{k}{2}}$.
\end{defi}
Now, with this definition of multiplication we have
$U_L = e(\beta)\cdot<R_{n-1},\beta S_{n-1}>$ if $L$ is the
plait closure of $\beta$.

\section{Action of colored braid group generators}\label{sec_action}

We choose the following generators of the colored braid group $B_{n,n}$ considered as a subgroup of $B_{2n}$ in the natural way:
$$
\begin{cases}
 a)&M_i=\s_{2i-1} \s_{2i}\s_{2i-1} = \s_{2i}\s_{2i-1}\s_{2i}, i=1,\dots,n-1 \\
 b)&N_i=\s_{2i}\s_{2i+1} \s_{2i} = \s_{2i+1} \s_{2i} \s_{2i+1}, i=1,\dots,n-1\\
 c)&P_i=\s_{2i}^2, i=1,\dots,n-1 \\
 d)&Q_i=\s_{2i-1}^2, i=1,\dots,n \\
\end{cases}
$$
Indeed, $M_i$ switch two odd strings, $N_i$ switch two even strings
and $P_i$ tangle odd and even strings. For a minimal set of generators and relations on them, see
\cite{generate}(Theorem 3).

\label{action}
Now we are ready to find the action of these generators of $B_{n,n}$
 on the generators of $H_1(D';\Theta)$.

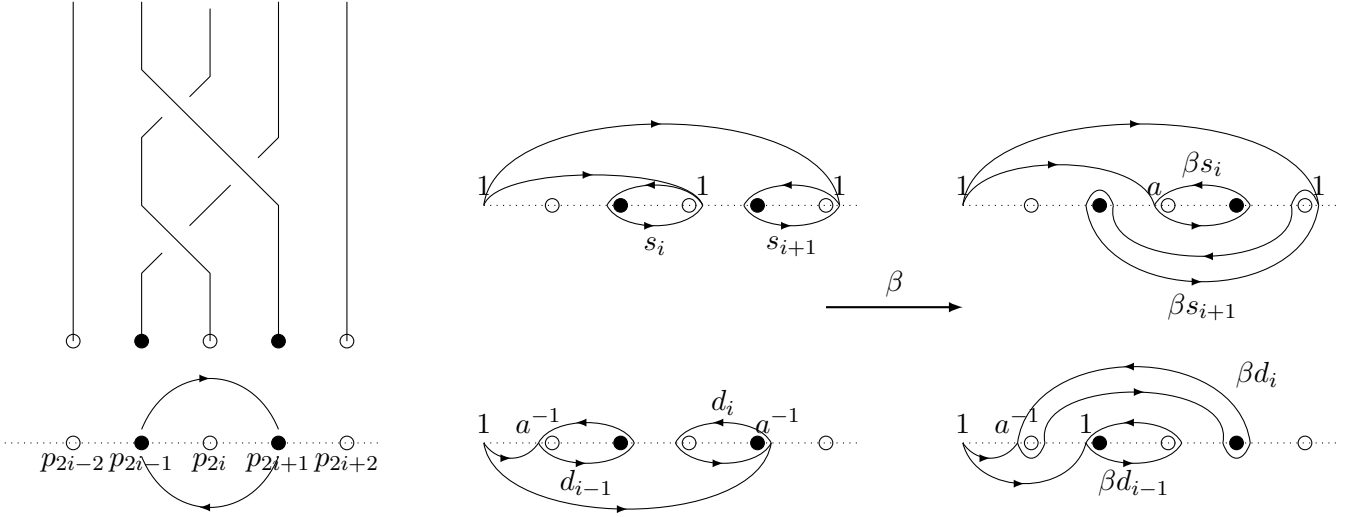
\begin{figure}[h]
\begin{center}
\begin{tikzpicture}[scale=0.9,>=latex]
\newcommand{\ra}{0.1};
\newcommand{\ba}{1.5};

\draw[dotted] (-5,0)--(0.5,0);
\draw (0,0) circle (\ra) ;
\draw[fill] (-1,0) circle (\ra) ;
\draw (-2,0) circle (\ra) ;
\draw[fill] (-3,0) circle (\ra) ;
\draw (-4,0) circle (\ra) ;

\draw (-4,0) node[below]{$p_{2i-2}$};
\draw (-3,0) node[below]{$p_{2i-1}$};
\draw (-2,0) node[below]{$p_{2i}$};
\draw (-1,0) node[below]{$p_{2i+1}$};
\draw (0,0) node[below]{$p_{2i+2}$};

\lin{-3}{0.2}{-1}{1}{0.4}{>}
\lin{-3}{-0.2}{-1}{-1}{0.4}{<}
\draw (0,\ba) circle (\ra) ;
\draw[fill] (-1,\ba) circle (\ra) ;
\draw (-2,\ba) circle (\ra) ;
\draw[fill] (-3,\ba) circle (\ra) ;
\draw (-4,\ba) circle (\ra) ;

\draw (-4,\ba)--++(0,5);
\draw (0,\ba)--++(0,5);
\draw (-1,\ba)--++(0,2)--++(-2,2)--++(0,1);
\draw (-2,\ba)--++(0,1)--++(-1,1)--++(0,1)--++(0.3,0.3);
\draw (-3,\ba)--++(0,1)--++(0.3,0.3);
\draw (-2.3,\ba+1.7)--++(0.6,0.6);
\draw (-1.3,\ba+2.7)--++(0.3,0.3)--++(0,2);
\draw (-2.3,\ba+3.6)--++(0.3,0.3)--++(0,1);
\renewcommand{\ba}{3.5}

\begin{scope}[xshift=1cm]
\draw[dotted] (1,\ba)--(6.5,\ba);
\draw (2,\ba) circle (\ra) ;
\draw[fill] (3,\ba) circle (\ra) ;
\draw (4,\ba) circle (\ra) ;
\draw[fill] (5,\ba) circle (\ra) ;
\draw (6,\ba) circle (\ra) ;
\draw (1,\ba) node[above] {$1$}; 
\draw (4.2,\ba) node[above] {$1$}; 
\draw (6.2,\ba) node[above] {$1$}; 

\lin{2.8}{\ba}{4.2}{0.4}{0.2}{<}
\lin{2.8}{\ba}{4.2}{-0.4}{0.2}{>}
\lin{4.8}{\ba}{6.2}{0.4}{0.2}{<}
\lin{4.8}{\ba}{6.2}{-0.4}{0.2}{>}
\lin{1}{\ba}{4.2}{0.6}{0.2}{>}
\lin{1}{\ba}{6.2}{1.6}{0.2}{>}
\draw (3.5,\ba-0.6) node{$s_i$};
\draw (5.5,\ba-0.6) node{$s_{i+1}$};
\draw[dotted] (8,\ba)--(13.5,\ba);
\draw (9,\ba) circle (\ra) ;
\draw[fill] (10,\ba) circle (\ra) ;
\draw (11,\ba) circle (\ra) ;
\draw[fill] (12,\ba) circle (\ra) ;
\draw (13,\ba) circle (\ra) ;
\draw (8,\ba) node[above] {$1$}; 
\draw (10.8,\ba) node[above] {$a$}; 
\draw (13.2,\ba) node[above] {$1$}; 

\lin{10.8}{\ba}{12.2}{0.4}{0.2}{<}
\lin{10.8}{\ba}{12.2}{-0.4}{0.2}{>}

\lin{12.8}{\ba}{10.2}{-1}{0.2}{>}
\linn{9.8}{\ba}{10.2}{0.3}{0.1}
\linn{12.8}{\ba}{13.2}{0.3}{0.1}
\lin{9.8}{\ba}{13.2}{-1.5}{0.2}{>}
\lin{8}{\ba}{10.8}{0.8}{0.2}{>}
\lin{8}{\ba}{13.2}{1.6}{0.2}{>}
\draw (11.5,\ba+0.6) node{$\beta s_i$};
\draw (11.5,\ba-1.5) node{$\beta s_{i+1}$};

\draw[->,thick] (6,2)--(8,2);
\draw (7,2) node[above]{$\beta$};

\renewcommand{\ba}{0}
\draw[dotted] (1,\ba)--(6.5,\ba);
\draw (2,\ba) circle (\ra) ;
\draw[fill] (3,\ba) circle (\ra) ;
\draw (4,\ba) circle (\ra) ;
\draw[fill] (5,\ba) circle (\ra) ;
\draw (6,\ba) circle (\ra) ;
\draw (1,\ba) node[above] {$1$}; 
\draw (1.8,\ba) node[above] {$a^{-1}$}; 
\draw (5.3,\ba) node[above] {$a^{-1}$}; 

\lin{1.8}{\ba}{3.2}{0.4}{0.2}{<}
\lin{1.8}{\ba}{3.2}{-0.4}{0.2}{>}
\lin{3.8}{\ba}{5.2}{0.4}{0.2}{<}
\lin{3.8}{\ba}{5.2}{-0.4}{0.2}{>}
\lin{1}{\ba}{1.8}{-0.3}{0.2}{>}
\lin{1}{\ba}{5.2}{-1.3}{0.2}{>}
\draw (2.5,\ba-0.6) node{$d_{i-1}$};
\draw (4.5,\ba+0.6) node{$d_{i}$};
\draw[dotted] (8,\ba)--(13.5,\ba);
\draw (9,\ba) circle (\ra) ;
\draw[fill] (10,\ba) circle (\ra) ;
\draw (11,\ba) circle (\ra) ;
\draw[fill] (12,\ba) circle (\ra) ;
\draw (13,\ba) circle (\ra) ;
\draw (8,\ba) node[above] {$1$}; 
\draw (9.8,\ba) node[above] {$1$}; 
\draw (8.8,\ba) node[above] {$a^{-1}$}; 
\lin{9.8}{\ba}{11.2}{0.4}{0.2}{<}
\lin{9.8}{\ba}{11.2}{-0.4}{0.2}{>}
\lin{11.8}{\ba}{9.2}{1}{0.2}{<}
\linn{8.8}{\ba}{9.2}{-0.3}{0.1}
\linn{11.8}{\ba}{12.2}{-0.3}{0.1}
\lin{8.8}{\ba}{12.2}{1.5}{0.2}{<}
\lin{8}{\ba}{9.8}{-0.8}{0.2}{>}
\lin{8}{\ba}{8.8}{-0.3}{0.2}{>}
\draw (10.5,\ba-0.6) node{$\beta d_{i-1}$};
\draw (12.3,\ba+1) node{$\beta d_{i}$};
\end{scope}

\end{tikzpicture}

\caption{The left picture illustrates the action of $M_i=\s_{2i-1}\s_{2i}\s_{2i-1},$ i.e., odd strings switch; $a=\tm$. We use the following convention: we read braids $\s_{i_1}\s_{i_2}\dots$ from left to right and apply them at the bottom, i.e. $\s_{i_1}$ being the lowest. When we act by a braid on the plane, time goes upwards on a braid.}\label{fig:fig6}
\end{center}
\end{figure}

a)$\beta=M_i$: see Fig.\ref{fig:fig6}. It is clear that $s_i\to \tm d_i$; breaking $\beta
s_{i+1}$ into basic semicircles we verify that $s_{i+1} \to s_i +
s_{i+1} - \tm d_i$. In the second row on the figure the action of $M_i$ on
$d_{i-1}, d_{i}$ is depicted.

$$
M_i=\s_{2i-1} \s_{2i}\s_{2i-1} \left\{
 \begin{array}{lcl}
   d_{i-1} \to d_{i-1} + d_i - \tm s_i \\
   s_i \to \tm d_i \\
   d_i \to \tm s_i \\
   s_{i+1} \to s_i + s_{i+1} - \tm d_i \\
  \end{array}
\right.
$$

b)$\beta=N_i$. The action of the following generator is derived from the previous one
by the substitutions $s_{i}\to d_{i}, d_{i-1}\to s_{i},
s_{i+1}\to d_{i+1}, d_i\to s_{i+1}, \tm \to \tp$.

$$
N_i=\s_{2i} \s_{2i+1} \s_{2i} \left\{
 \begin{array}{lcl}
   s_i \to s_i + s_{i+1} - \tp d_i \\
   d_{i} \to \tp s_{i+1} \\
   s_{i+1} \to \tp d_{i} \\
   d_{i+1} \to d_i + d_{i+1} - \tp s_{i+1} \\
 \end{array}
\right.
$$

c)$\beta=P_i$: see Fig.\ref{fig:mer7}
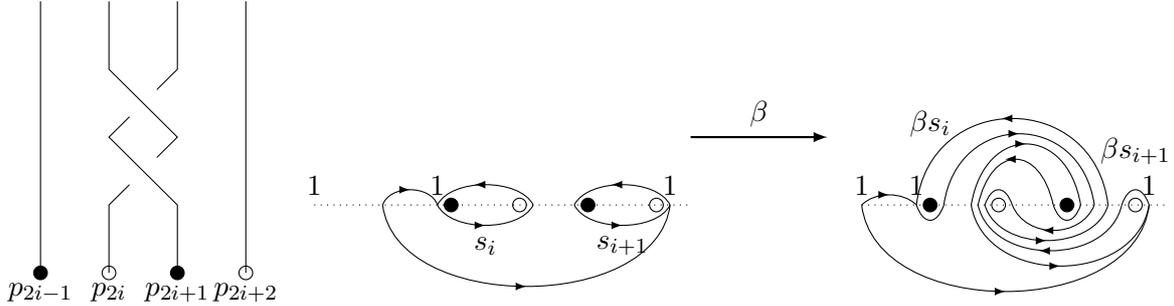
\begin{figure}[h]
\begin{center}

\begin{tikzpicture}[scale=0.9,>=latex]
\newcommand{\ra}{0.1};
\newcommand{\ba}{1.5};
\draw (-3,\ba) node[below]{$p_{2i-1}$};
\draw (-2,\ba) node[below]{$p_{2i}$};
\draw (-1,\ba) node[below]{$p_{2i+1}$};
\draw (0,\ba) node[below]{$p_{2i+2}$};
\draw (0,\ba) circle (\ra) ;
\draw[fill] (-1,\ba) circle (\ra) ;
\draw (-2,\ba) circle (\ra) ;
\draw[fill] (-3,\ba) circle (\ra) ;
\draw (0,\ba)--++(0,4);
\draw (-1,\ba)--++(0,1)--++(-1,1)--++(0.3,0.3);
\draw (-2,\ba)--++(0,1)--++(0.3,0.3);
\draw (-3,\ba)--++(0,4);

\draw (-1.3,\ba+1.7)--++(0.3,0.3)--++(-1,1)--++(0,1);
\draw (-1.3,\ba+2.7)--++(0.3,0.3)--++(0,1);

\renewcommand{\ba}{2.5}

\draw[dotted] (1,\ba)--(6.5,\ba);
\draw[fill] (3,\ba) circle (\ra) ;
\draw (4,\ba) circle (\ra) ;
\draw[fill] (5,\ba) circle (\ra) ;
\draw (6,\ba) circle (\ra) ;
\draw (1,\ba) node[above] {$1$}; 
\draw (2.8,\ba) node[above] {$1$}; 
\draw (6.2,\ba) node[above] {$1$}; 

\lin{2.8}{\ba}{4.2}{0.4}{0.2}{<}
\lin{2.8}{\ba}{4.2}{-0.4}{0.2}{>}
\lin{4.8}{\ba}{6.2}{0.4}{0.2}{<}
\lin{4.8}{\ba}{6.2}{-0.4}{0.2}{>}
\lin{2}{\ba}{2.8}{0.3}{0.1}{>}
\lin{2}{\ba}{6.2}{-1.6}{0.2}{>}
\draw (3.5,\ba-0.6) node{$s_i$};
\draw (5.5,\ba-0.6) node{$s_{i+1}$};
\draw[dotted] (9,\ba)--(13.5,\ba);
\draw[fill] (10,\ba) circle (\ra) ;
\draw (11,\ba) circle (\ra) ;
\draw[fill] (12,\ba) circle (\ra) ;
\draw (13,\ba) circle (\ra) ;
\draw (9,\ba) node[above] {$1$}; 
\draw (9.8,\ba) node[above] {$1$}; 
\draw (13.2,\ba) node[above] {$1$}; 

\linn{12.8}{\ba}{13.2}{0.3}{0.1}
\lin{10.7}{\ba}{12.8}{-0.9}{0.2}{<}
\lin{10.6}{\ba}{13.2}{-1.2}{0.2}{>}
\lin{10.7}{\ba}{11.8}{0.9}{0.2}{<}
\lin{10.6}{\ba}{12.2}{1.2}{0.2}{>}
\linn{11.8}{\ba}{12.2}{-0.3}{0.1}

\linn{10.8}{\ba}{11.2}{0.3}{0.1}
\lin{10.8}{\ba}{12.6}{-0.7}{0.2}{>}
\lin{11.2}{\ba}{12.4}{-0.5}{0.2}{<}
\lin{9.8}{\ba}{12.6}{1.7}{0.2}{<}
\lin{10.2}{\ba}{12.4}{1.4}{0.2}{>}
\linn{9.8}{\ba}{10.2}{-0.3}{0.1}
\lin{9}{\ba}{9.8}{0.2}{0.2}{>}
\lin{9}{\ba}{13.2}{-1.7}{0.2}{>}
\draw (10,\ba+1.2) node{$\beta s_i$};
\draw (13,\ba+0.8) node{$\beta s_{i+1}$};
\draw[->,thick] (6.5,3.5)--(8.5,3.5);
\draw (7.5,3.5) node[above]{$\beta$};
\end{tikzpicture}
\caption{$P_i=\s_{2i}^2,$ double tangling of two neighbor strings.}\label{fig:mer7}
\end{center}
\end{figure}
$$
P_i=\s_{2i}^2 \left \{
 \begin{array}{lcl}
   s_i \to s_i + \td d_i \\
   d_i \to d_i \\
   s_{i+1} \to s_{i+1} - \td d_i  \\
 \end{array}
\right.
$$

d)$\beta=Q_i$: The difference with c) is the same as the difference
between b) and a): we can apply the inverse of the substitution used in b).
$$
Q_i=\s_{2i-1}^2 \left\{
 \begin{array}{lcl}
   d_{i-1} \to d_{i-1} -\td s_i \\
   s_i \to s_i \\
   d_i \to d_i + \td s_i  \\
 \end{array}
\right.
$$

\begin{rem}
In fact, we described the explicit action of $M_i,N_i,P_i,Q_i\in B_{n,n}$ on the
$2n$-dimensional space, generated by the basis vectors labelled
$[d_i],[s_i],i=1,...,n$, this action clearly descends to $H_1(D',\Theta)$.
This is similar to the difference between the Burau
representation and the reduced Burau representation.
\end{rem}

Let us calculate $U_L$ for a number of examples, directly from our definition.
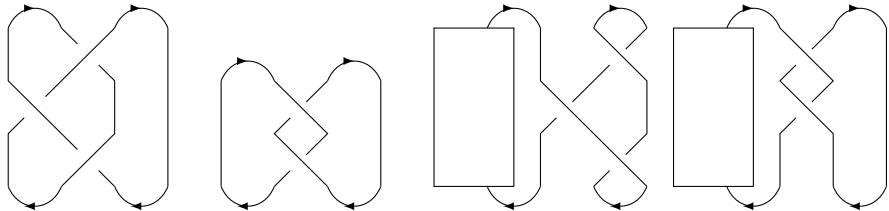
\begin{figure}[h]
\begin{center}
\begin{tikzpicture}[scale=0.7,>=latex]
\draw (0,0)--++(0,1)--++(0.3,0.3);
\lin{0}{0}{1}{-0.5}{0.2}{<}
\lin{2}{0}{3}{-0.5}{0.2}{<}

\draw (1,0)--++(1,1)--++(0,1)--++(-0.3,0.3);
\draw (1.3,0.7)--++(-0.3,0.3)--++(-1,1)--++(0,1);
\lin{0}{3}{1}{0.5}{0.2}{>}
\lin{2}{3}{3}{0.5}{0.2}{>}

\draw (0.7,1.7)--++(1.3,1.3);
\draw (2,0)--++(-0.3,0.3);
\draw (3,0)--(3,3);
\draw (1.3,2.7)--(1,3);

\begin{scope}[xshift=4cm]
\lin{0}{0}{1}{-0.5}{0.2}{<}
\lin{2}{0}{3}{-0.5}{0.2}{<}
\lin{0}{2}{1}{0.5}{0.2}{>}
\lin{2}{2}{3}{0.5}{0.2}{>}
\draw (0,0)--(0,2);
\draw (3,0)--(3,2);
\draw (1,0)--(1.3,0.3);
\draw(1.6,0.6)--(2,1)--(1,2);
\draw (2,0)--(1,1)--(1.3,1.3);
\draw (1.6,1.6)--(2,2);
\end{scope}

\begin{scope}[xshift=9cm]
\lin{0}{0}{1}{-0.5}{0.2}{<}
\lin{2}{0}{3}{-0.5}{0.2}{<}
\lin{0}{3}{1}{0.5}{0.2}{>}
\lin{2}{3}{3}{0.5}{0.2}{>}

\draw(0.5,0)--(-1,0)--(-1,3)--(0.5,3)--(0.5,0);
\draw (1,0)--(1,1)--(1.3,1.3);
\draw (1.6,1.6)--(2.3,2.3);
\draw (2.6,2.6)--(3,3);
\draw (2,0)--(2.3,0.3);
\draw (2.6,0.6)--(3,1)--(3,2)--(2,3);
\draw (3,0)--(1,2)--(1,3);

\end{scope}
\begin{scope}[xshift=13.5cm]
\lin{0}{0}{1}{-0.5}{0.2}{<}
\lin{2}{0}{3}{-0.5}{0.2}{<}
\lin{0}{3}{1}{0.5}{0.2}{>}
\lin{2}{3}{3}{0.5}{0.2}{>}
;
\draw (3,0)--(3,3);
\draw (1,0)--(1,1)--(1.3,1.3);
\draw(1.6,1.6)--(2,2)--(1,3);
\draw (2,0)--(2,1)--(1,2)--(1.3,2.3);
\draw (1.6,2.6)--(2,3);
\draw (0.5,0)--(-1,0)--(-1,3)--(0.5,3)--(0.5,0);
\end{scope}

\end{tikzpicture}

\caption{Left trefoil knot, Hopf link (Examples \ref{ex1},\ref{ex2}), adding a ``trivial component'', hanging a component (Examples \ref{examplest},\ref{ex4}).} \label{void}
\end{center}
\end{figure}

\begin{ex} \label{ex1}The left trefoil can be presented as the plait closure of the braid
$T = \s_2^{-1}\s_1\s_2^{-1} =
\s_{2}^{-2}\cdot\s_1\s_2\s_1\cdot\s_2^{-2} = P_1^{-1}M_1P_1^{-1}\in B_{2,2}$. Using Definition \ref{def_e} we see that $T_1=\s_1,T_2=1$ and $e(T)=\tp$.
\end{ex}
A direct calculation gives us 

$$
T \left \{
 \begin{array}{lcl}
   s_1 \xrightarrow{P_1^{-1}} s_1 - \td d_i \xrightarrow{M_1} \tm d_1 + (1-t^{-1})s_1 \xrightarrow{P_1^{-1}} \\\\
   \ \ \  \tm d_1 + (1-t^{-1})s_1 - (1-t^{-1})\td d_1\\
   \\\\
   s_2 \xrightarrow{P_1^{-1}} s_2 + \td d_i \xrightarrow{M_1} s_1 + s_2 - \tm d_1 - (1-t^{-1})s_1 \xrightarrow{P_1^{-1}} \\\\
   \ \ \   s_1 + s_2 - \tm d_1 - (1-t^{-1})s_1 + (1-t^{-1})\td d_1 \\
 \end{array}
\right.
$$

So, $e(T)\cdot r_1\times T([s_1]) = \tp \cdot r_1\times (\tm [d_1] +
(1-t^{-1})[s_1]+ \td(t^{-1}-1)[d_1]) = t + t^{-1} -1$, as it should be.

\begin{ex}\label{ex2} The Hopf link can be presented as the plait closure of the braid $H 
=P_1=\s_2^2\in B_{2,2}$. Then,  $H_1=H_2=1, e(\s_2^2) = 1$, $r_1\times H([s_1]) = r_1\times ([s_1] + \td
[d_1]) = \td$ and we get the Conway polynomial of the Hopf link.
\end{ex}

\section{Auxiliary facts}\label{sec_aux}
In this section we fix a particular $\beta\in B_{n,n}$. Let $B$ be
the matrix of the action of $\beta$ on $W$
with respect to the basis $\{[s_1], [d_1],[s_2],[d_2],\dots,[d_n]\}$ (in that order).
We consider the submatrix $B'$  which consists of all even rows and odd columns of
$B$. Note that $r_i\times \beta s_j$ is the entry of $B$ at the
intersection of $2i$-th row and $(2j-1)$-th column, and, therefore it is
the entry of $B'$ at the intersection of $i$-th row and $j$-th column. 

\begin{prop}$U_\beta$ is the determinant of the submatrix of $B'$, which
consists of first $n-1$ rows and first $n-1$ columns.\qed
\end{prop}
The submatrix considered is that determined by the first $n-1$ even
rows and first $n-1$ odd columns in $B$. Then, in the definition of $U_\beta$ we see $|\{r_i\}|=n-1,
|\{[d_i]\}|=n-1$, therefore they must be paired. We have $r_1\times
[d_j]=0(1\leqslant j\leqslant n-1, j\ne 1)$, therefore $r_1$
corresponds to $[d_1]$ (i.e. the first even column in $B$). Hence $r_2$
corresponds to $[d_2]$, because $[d_1]$ is already in a pair, etc. Now it is easy to see
that the definition of $U_\beta$ is exactly the definition of the considered minor.

Note, that $B'$ is an $n\times n$ matrix.
\begin{theo}The determinant of each minor of $B'$ with $n-1$ rows 
and $n-1$ columns is $\pm U_\beta$.
\end{theo}

The proof consists of two following lemmata.

\begin{defi}
A matrix, which has zero sum of elements in each row and each column, is called {\it pseudostochastic}.
\end{defi}

\begin{lemma}\label{lemma_pse} The submatrix $B'$  which consists of all even rows and odd columns of
$B$ is pseudostochastic.
\end{lemma}

\begin{proof} It follows from the previous section that 
$\sum_{i=1}^n [s_i]$ is invariant under the action of $B_{n,n}$.
Therefore $r_k\times \sum_{i=1}^n [\beta s_i] = r_k\times \sum_{i=1}^n [s_i] = 0$ and the fact about columns is proven.

Let us prove the condition about rows. For each $k$ we have $\sum_{i=1}^n r_i
\times [s_k] = 0, \sum_{i=1}^n r_i \times [d_k] = 0$. Therefore $\sum_{i=1}^n r_i=0$ in
$H^1(D';\Theta^*)$ and hence $\sum_{i=1}^n r_i \times [\beta s_k] = 0$.
\end{proof}

\begin{lemma}\label{lemma_min}All the first minors of a pseudostochastic matrix $C$ are equal modulo sign change.
\end{lemma}
\begin{proof} It follows from the definition of a pseudostochastic matrix that rows of $C$ (as vectors) lie in the hyperplane
$\sum x_i=0$ and their center of mass is 0. A first minor of $C$ is the oriented
volume of a simplex consisting of $n-1$ vectors, projected onto a coordinate hyperplane.
All such volumes are equal because $0$ is the center of mass of these
$n$ points, and all the angles between
coordinate hyperplanes and the hyperplane $\sum x_i=0$ are also equal.
\end{proof}

\begin{proof}[Alternative proof] Consider the matrix $C'$ which is
  complementary to $C$, i.e. the elements of $C'$ are the first minors of $C$. We have $$CC' =
det(C)I = 0.$$ Without loss of generality we suppose that at least one
first minor of $C$ is not zero. Then the rows of $C$, understood as vectors, lie in the hyperplane
$\sum_{i=1}^n x_i = 0$ and span it. Therefore, the rows of $C$ generate the vectors $e_{i+1}
- e_{i}, i =1,\dots,n-1$ where $\{e_j\}_{i=1}^n$ are the basis vectors. From the fact that $(e_{i+1} -
e_i)C' = 0$ we conclude that columns of $C'$ are the same modulo sign-change; it follows from $C'C =
0$ that the same is true for the rows.
\end{proof}

\begin{defi}\label{st}We say that $st^*(\beta)\in B_{n+1,n+1}$ is obtained from $\beta\in
B_{n,n}$ via adding a trivial component (Fig.\ref{void}, third picture) if $st^*(\beta)$ is constructed by taking
$2n+2$ strings, applying $\s_{2n+1}\s_{2n}\s_{2n+1}$ at the very bottom of the
obtained braid, and then applying $\beta$ on the first $2n$ strings.
\end{defi}

Similarly to Examples \ref{ex1},\ref{ex2} we use the notation of Definition \ref{def_e}. With the natural identification of $B_n$ as a subgroup of $B_{n+1}$ with the last string trivial we have $st^*(\beta)_1=\beta_1,st^*(\beta)_2=\sigma_{n+1}\beta_2, e(st^*(\beta))=\tm e(\beta)$. In order to prove that ``adding a trivial component'' (third picture in Fig.~\ref{void}) does not change our invariant $U_L$ we now calculate $r_{n+1}\times [s_{n+1}]$ for later use in Corollary~\ref{cor_trivial}. 
\begin{ex}
\label{examplest} We will prove that $(r_{n+1}\times [s_{n+1}])\cdot e(st^*(\beta))=e(\beta)$. Since $\s_{2n+1}\s_{2n}\s_{2n+1}=N_n$ acts as $s_{n+1}\xrightarrow{N_n}
\tp d_n$, while $r_{n+1}\times [s_{n+1}]=\tp, r_{n+1}\times [s_{i}]=0$ for
$i=1,...,n-1$, the statement follows from the fact that $(r_{n+1}\times
[s_{n+1}])\cdot (\tm e(\beta)) = e(\beta)$.
\end{ex}

\begin{cor}
\label{cor_trivial}
For each braid $\beta\in B_{n,n}$ the equality $e(\beta)U_\beta(t) = e(st^*(\beta))U_{st^*(\beta)}(t)$ holds.
\end{cor}
\begin{proof}
Indeed, thanks to Example~\ref{examplest} and Lemmata \ref{lemma_pse},\ref{lemma_min} we see that $$e(\beta)U_\beta=e(\beta)((n-1)!)^2(r_1\wedge r_2\wedge \dots
\wedge r_{n-1})\times ([\beta s_1]\wedge \dots \wedge [\beta s_{n-1}])= $$
$$=e(st^*(\beta))(n!)^2(r_1\wedge r_2\wedge \dots
\wedge r_{n-1}\wedge r_{n+1})\times ([\beta s_1]\wedge \dots \wedge [\beta s_{n-1}]\wedge[\beta s_{n+1}])
=e(st^*(\beta))U_{st^*(\beta)}.$$
 
\end{proof}

\begin{defi} {\it Hanging a circle} is adding to a braid
$\beta\in B_{n,n}$ two new strings with numbers $2n+1, 2n+2$ with applied 
$\s_{2n}^2$ in the very top of the obtained braid, see Fig.\ref{void}, forth picture.
\end{defi}

\begin{ex}\label{ex4}Let $\beta$ be a colored braid, $\beta'$ is obtained from
$\beta$ by hanging a circle. Let $L,L'$ be the plait closures 
of $\beta,\beta'$ correspondingly. Then $U_{L'} =
U_{L}\cdot (\tp - \tm)$.
\end{ex}
\begin{proof} $U_{\beta} = (n-1)!(r_1\wedge r_2\wedge \dots
\wedge r_{n-1})\times ([\beta s_1]\wedge \dots \wedge [\beta s_{n-1}])$. Clearly $e(\beta')/e(\beta)=-1$. It follows from Corollary \ref{cor_trivial} that 
in order to obtain $U_{\beta'}$ we can add $r_{n+1}\times [s_{n+1}]$ to $U_{\beta}$. Then
$$r_{n+1}\times\s_{2n}^2 ([s_{n+1}]) = r_{n+1}\times([s_{n+1}] + (\tp -
\tm)[d_n]) = -(\tp - \tm).$$  For $i=1,...,n-1$ we have $r_{i}\times\s_{2n}^2 ([s_{n+1}]) = 0$. Finally,  the statement follows from $$r_{n+1}\times\s_{2n}^2 ([s_{n+1}]) \cdot e(\beta')/e(\beta) = (\tp - \tm).$$
That agrees with the calculation for the Hopf link.
\end{proof}

For a braid $\eta\in B_{n,n}$ we denote by $\tilde\eta$ the action on
$W$ induced by $\eta$.

\begin{defi}
We denote by $W_s$  the subspace of $W$
spanned by $[s_1],\ldots, [s_n]$ and by $W_d$ the subspace spanned
by $[d_1],\ldots, [d_n]$.
\end{defi}

\begin{cor}
\label{corollary}{a) Let $\eta_s,\eta_d$ be braids such that
$\tilde\eta_s,\tilde\eta_d\colon W \to W$ satisfy
$$\eta_s([s_i]) = [s_i] (i=1,\dots,n), \eta_d([d_i]) = [d_i] + ss_{i,1}(i=1,\dots,n),\eta_d([s_i]) =
ss_{i,2}(i=1,\dots,n)$$ where $ss_{i,*}$ are vectors in $W_s$. Then $U_{\beta} =
U_{\eta_s\beta} = U_{\beta \eta_d}$.

b)\label{important} More generally,  suppose that $W_s$ and $W_d$ are invariant under the action of $\tilde\eta_s$, and 
$W_s$ is invariant under the action of $\tilde\eta_d$. Then
$$U_{\eta_s\beta} = det({\tilde{\eta}_s}{|_{W_s}})\cdot U_{\beta}, U_{\beta
\eta_d} = det({\tilde{\eta}_d}{|_{W_d}})\cdot U_{\beta}.$$}
\end{cor}
\begin{proof} a)It is clear that the transformations  $\tilde\eta_s,\tilde\eta_d$ do not change the submatrix
$B'$. b) In the above hypothesis the minors of $B'$ are multiplied by
the determinant of the matrix of $\tilde\eta_s$ (resp.  $\tilde\eta_d$)
restricted to $W_s$ (resp. $W_d$). In the first case we have
$det({\tilde{\eta}_s}{|_{W_s}})$, in the second case we have
 $det({\tilde{\eta}_d}{|_{W_d}})$.
\end{proof}

\section{Proof of Theorem 1}\label{sec_proof}

\begin{defi} Let $K_{2n}$ be the subgroup of $B_{2n}$
generated by

$$
\begin{array}{lcl}
  \s_1, \\
  A=\s_2 \s_1^2 \s_2, \\
  A_i=\s_{2i}\s_{2i-1}\s_{2i+1}\s_{2i},  i=1,\dots,n-1 \\
\end{array}
$$
\end{defi}

\begin{defi}
 Given a braid $\eta\in B_{2n}$ with $2n$ strings, the {\it stabilization}
 $st(\eta)\in B_{2n+2}$ or {\it adding a
 ``trivial loop''} is the addition of two new strings with numbers $2n+1,2n+2$ and applying $\s_{2n+1}$ at the bottom of the obtained braid.
\end{defi}
It is easy to see, that neither stabilization nor elements of $K_{2n}$ change the topological type of the plait closure of a braid.

Let links $L_1,L_2$ be the plait closures of two braids  $\beta_1\in B_{2n_1}, \beta_2\in B_{2n_2}$.

\begin{lemma} (J. Birman, Theorem 1', \cite{Birman}). \label{lemma_birman}The links $L_1,L_2$ are equivalent if and only if after adding a number of ``trivial loops'' to each component of $L_1$ and $L_2$;
   $\beta_1 \to \beta_1'\in B_{2n}, \beta_2 \to \beta_2'\in
B_{2n}$, these braids $\beta_1',\beta_2'$ lie in the same coset of $K_{2n}$, i.e.

$$\beta_1' = g\beta_2' h, \beta\in B_{2n}, g,h\in K_{2n}.$$
\end{lemma}

This formulation is not very accurate since we defined stabilization
at the right end of a braid and here we are required to perform it for
each component of links $L_1,L_2$. Still, we can repair that by the
following rewording. Given two braids $\beta_1,\beta_2$ we allow the
following operations: applying the stabilization $\eta\to st(\eta)\in
B_{2k}$, then adding elements of
$K_{2k}$ from the left and from the right. If the plait closures of
$\beta_1,\beta_2$ are equivalent, then there are sequences of such
operations which produce the same braid starting from $\beta_1\in
B_{2n_1}$ and $\beta_2\in B_{2n_2}$.   

We are going to use this result for braids in $B_{n,n}$. For that we
need to care about coloring of the braids since elements of $K_{2n}$
do not respect the coloring. Our goal is to prove that
$U_{L_1}=U_{L_2}$ if $L_1$ is equivalent to $L_2$. For
that, we replace the
operation $st$ with $st^*$ (Def. \ref{st}). Lemma \ref{lemma_birman} holds with $st^*$ instead of $st$. Then, we can assume that our braids are oriented and we perform plait
closure with respect to the orientations of the strings. Indeed, $A,A_i$ respect the orientation and $\s_1$ change two strings in the same pair, which we link while performing the plait closure.

Now, if we have $\beta_1' = g\beta_2' h, \beta\in B_{2n}, g,h\in
K_{2n}$ at the end, we can apply a number of $\s_{2i+1}$ on the top
and bottom of the braids such that the strings oriented
down connect the points with odd numbers, like in Fig.\ref{fig}. So, we get 
$m_1\beta_1m_1'=m_2\beta_2m_2'$ where $m_i,m_i'\in K_{2n}, i=1,2$ respect the
coloring. 

\begin{cor}
Hence, in order to prove Theorem \ref{th_main} it is enough to prove that $$e(m_1\beta_1m_1')U_{m_1\beta_1m_1'}=e(m_2\beta_2m_2')U_{m_2\beta_2m_2'}$$ if $m_1\beta_1m_1'=m_2\beta_2m_2'$ in $B_{2n}$, and $\beta_1\in B_{n_1,n_1},\beta_2\in B_{n_2,n_2}$, and $m_i,m_i'\in K_{2n}, i=1,2$ respect the coloring. 
\end{cor}

\begin{defi} Let $R_n$ be the subgroup of $B_{2n}$ generated by the following elements:
$$
\begin{cases}
a)&\s_{2i-1}^2, i =1,\dots, n \\
b)&\s_1 A \s_1, \\
c)&A, \\
d)&A_i, i = 1,\dots, n-1 \\
e)&\s_{2i-1}A_i\s_{2i+1}, \s_{2i+1}A_i\s_{2i-1}, i = 1,\dots, n-1.\\
\end{cases}
$$
\end{defi}
Clearly, elements of $R_n$ respect the coloring.

For a sequence  $\s_{2j_1+1}^{n_1}A^{k_1}\s_{2j_2+1}^{n_2}A_{i_1}^{k_2}\s_{2j_3+1}^{n_3}A^{k_3}\s_{2j_4+1}^{n_4}A_{i_2}^{k_4}\dots,$ where all the indices are integers we define its length as $|k_1|+|k_2|+\dots.$

\begin{lemma}\label{lemma_color} If $g\in K_{2n}$ preserves the coloring, then $g\in R_n$.
\end{lemma}

\begin{proof} We can  
present $g$ as a product $l=\s_{1}^{n_1}A^{k_1}\s_{1}^{n_2}A_{i_1}^{k_2}\s_{1}^{n_3}A^{k_3}\s_{1}^{n_4}A_{i_2}^{k_4}\dots$. Now we will nip off generators of $R_n$ from the left of $g$, decreasing the length of 
this sequence. We also use that $\s_1$ commutes with $A_i,i\geq 2$, and $\s_3$ commutes with $A_i,i\geq 3$, and $\s_{2k-1}, k =3,\dots, n$ commutes with $A$ and $A_i, i\geq k+1, i\leq k-2$. Using these commuting relations we try to push the letters $A, A_i$ towards the left end of the string. So, starting with $g$ we can always nip off $k_1$ copies of $A$, then $k_2$ copies of $A_{i_1}$ etc. Note, that we allow negative powers and a nipping may increase the powers of $\s_{2i+1}$ in the string. At the end of the process, all $A$ and $A_i$ terms have been nipped off, leaving only a sequence of the type $\prod_i\s_{2k_i+1}$. Using the fact that $\s_{2i+1}$ commutes with $\s_{2j+1}$ we can nip off all the generators of the type $\s_{2i+1}^2$ and end up with a sequence $\prod_i\s_{2k_i+1}$ where all the indices are distinct numbers. Since $g$ preserves colorings the latter product must be $1$ in $K_{2n}$.
\end{proof}

\begin{lemma}\label{lemma_coset} For each element $\gamma\in R_n$ and each $\beta\in B_{n,n}$
  we have $e(\gamma\beta)U_{\gamma\beta}=e(\beta \gamma)U_{\beta \gamma}=e(\beta)U_\beta$.
\end{lemma}

\begin{proof} In what follows we
use the explicit formulae for the action (Section \ref{action}) and Corollary \ref{important}.

a)$\gamma=\s_{2i-1}^2$, applied in the bottom doesn't change $\beta s_i$
at all, and if we apply $\gamma$ in the very top, it adds some amount of $s_i$ to $d_i, d_{i-1}$, which has no influence on  $U_\beta$.

b) $\gamma=\s_1 A \s_1 = \s_1 \s_2 \s_1 \cdot \s_1 \s_2 \s_1$, acts on the standard basis for $W$ by

$$
\s_1 A \s_1 \left \{
 \begin{array}{lcl}
   d_0 \to d_0 + (1 - t^{-1})d_1 \\
   s_1 \to t^{-1} s_1 \\
   d_1 \to t^{-1} d_1 \\
   s_2 \to s_2 + s_1(1-t^{-1})  \\
 \end{array}
\right.
$$

The matrix of the action on $d_0,d_1$ (in which we are interested 
when we add $\s_1 A \s_1$ to the very top of the braid) is the same as the
matrix of the action on $s_2,s_1$ which corresponds to adding to to the bottom.

$$ det(\widetilde{\beta\gamma}|_{W_d})=det(\widetilde{\gamma\beta}|_{W_d})=det(\tilde\gamma|_{W_s})
\begin{vmatrix}
  1  & 1-t^{-1}\\
  0  & t^{-1}  \\
\end{vmatrix}
$$

Matrix determinant equals $t^{-1}$, therefore
$U_{\gamma\beta}=U_{\beta\gamma}=U_{\beta}\cdot t^{-1}$. The new term
$t^{-1}$ is compensated by $e(\gamma\beta)/e(\beta)=t$.

c)$\gamma=A = \s_2 \s_1^2 \s_2 = \s_1\s_2\s_1 \cdot \s_2^{-2} \cdot
\s_1\s_2\s_1$.

$$
A \left \{
 \begin{array}{lcl}
   d_0 \to d_0 + (1-t^{-1})d_1 + (t^{-\frac{3}{2}} - \tm)s_1 \\
   s_1 \to t^{-1}s_1 \\
   d_1 \to t^{-1}d_1 -(t^{-\frac{3}{2}} - \tm)s_1 \\
   s_2 \to s_2 + (1-t^{-1})s_1  \\
 \end{array}
\right.
$$
Again, as in b), we apply Corollary \ref{important}.

d)$\gamma=A_i = \s_{2i}\s_{2i-1}\s_{2i+1}\s_{2i} =
\s_{2i-1}\s_{2i}\s_{2i-1} \cdot
\s_{2i}^{-2}\cdot\s_{2i}\s_{2i+1}\s_{2i}$.

$$
A_i \left \{
 \begin{array}{lcl}
   d_{i-1} \to d_{i-1} + d_i- \tm s_i \\
   s_i \to s_{i+1} \\
   d_i \to \tp - d_i +\tm s_i + \tp s_{i+1} \\
   s_{i+1} \to s_i  \\
   d_{i+1} \to d_i + d_{i+1} - \tp s_{i+1}
 \end{array}
\right.
$$

If we add this transformation to the very bottom of the braid such that $s_i$ switch with
$s_{i+1}$ then this changes $U_\beta$ by multiplication by $-1$. Consider the case when
we insert $A_i$ in the top of the braid, the corresponding matrix is
multiplied by a matrix whose action on the vectors
$d_{i-1},d_i,d_{i+1}$ is given by the matrix

$$
\begin{vmatrix}
   1 & 1 & 0 \\
   0 & -1& 0 \\
   0 & 1 & 1 \\
\end{vmatrix}
$$

Its determinant equals $-1$. Therefore
$U_{\beta\gamma}=U_{\gamma\beta} = - U_{\beta}$, while $e(\beta\gamma)/e(\beta)=-1$.

e) Both elements $\s_{2i-1}A_i\s_{2i+1}, \s_{2i+1}A_i\s_{2i-1}$ have the same action modulo Corollary \ref{important}, we
write all the details for only one of them:

$$
\s_{2i-1}A_i\s_{2i+1} \left \{
 \begin{array}{lcl}
   d_{i-1} \to d_{i-1} -\tm s_i + d_i \\
   s_i \to s_{i+1} \\
   d_i \to \tp s_i + \tp s_{i+1} - d_i \\
   s_{i+1} \to s_i  \\
   d_{i+1} \to d_i + d_{i+1} - \tp s_{i+1}
 \end{array}
\right.
$$

Then the same as in d).

Finally we see that $e(\beta')/e(\beta)$  kills all the additional coefficients
appearing in the above considerations. 

\begin{figure}[h]
\begin{center}
\begin{tikzpicture}[scale=0.3,>=latex]
\draw (0,0)--++(0,2)--++(0.3,0.3);
\draw (0.7,2.7)--++(0.3,0.3)--++(-1,1)--++(0,2);
\draw (1,0)--++(0,1)--++(0.3,0.3);
\draw (1.7,1.7)--++(0.3,0.3)--++(0,2)--++(-1,1)--++(0,1);

\draw (2,0)--++(0,1)--++(-2,2)--++(0.3,0.3);
\draw (0.7,3.7)--++(0.6,0.6);
\draw (1.7,4.7)--++(0.3,0.3)--++(0,1);
\draw (3,0)--++(0,6);
\draw (1.5,-1) node {$e(\beta)=t$};
\draw [decoration={markings, mark=at position 0.5 with {\arrow{>}}},postaction={decorate}] (2,5.5)--(2,6);
\draw [decoration={markings, mark=at position 0.5 with {\arrow{>}}},postaction={decorate}] (0,5.5)--(0,6);

\begin{scope}[xshift=6cm]

\draw [decoration={markings, mark=at position 0.5 with {\arrow{>}}},postaction={decorate}] (2,5.5)--(2,6);
\draw [decoration={markings, mark=at position 0.5 with {\arrow{>}}},postaction={decorate}] (0,5.5)--(0,6);
\draw (1.5,-1) node {$e(\beta)=-1$};

\draw (0,0)--++(0,2)--++(0.3,0.3);
\draw (0.7,2.7)--++(0.3,0.3)--++(0,1)--++(0.3,0.3);
\draw (1.7,4.7)--++(0.3,0.3)--++(0,1);
\draw (1,0)--++(0,1)--++(0.3,0.3);
\draw (1.7,1.7)--++(0.3,0.3)--++(0,1)--++(0.3,0.3);
\draw (2.7,3.7)--++(0.3,0.3)--++(0,2);
\draw (2,0)--++(0,1)--++(-2,2)--++(0,3);
\draw (3,0)--++(0,3)--++(-2,2)--++(0,1);

\end{scope}
\end{tikzpicture}
\caption{Examples of $e(\beta)$.}
\end{center}
\end{figure}
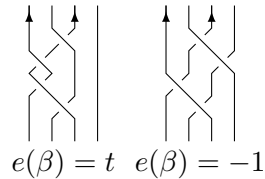

Indeed, if we switch two odd strings we multiply by
$\tp$, if we switch two even strings we multiply by $-\tm$. 
In the cases b),c) we get $\tp\cdot \tp = t$, and for
d),e) we get $\tp\cdot(-\tm) = -1$, which coincides with the
definition of $e(\beta)$.
\end{proof}

In Theorem \ref{th_main}  we need to prove that if  $L$ is the plait closure of $\beta$, then
  $U_L$ does not depend on $\beta$, and depends only on $L$ .

\begin{proof}[Proof of Theorem \ref{th_main}] Suppose $L$ is presented as the plait closures of
braids $\beta_1, \beta_2$. Using result of Birman
we add to $\beta_1, \beta_2$ appropriate number of ``trivial components''
which do not change $U$ thanks to Corollary \ref{cor_trivial}, and
we get braids $\beta_1',
\beta_2'\in B_{n,n}$ which lie in the same coset of $K_{2n}$.
Lemma \ref{lemma_color} asserts that if $\beta_1',\beta_2'$ are colored braids
with equal plait closures then $\beta_1', \beta_2'\in B_{n,n}$ lie in the same
coset by subgroup $R_{n}$, and $V_L$ depends only on such a coset by
the Lemma \ref{lemma_coset}. 
\end{proof}

\section{Skein relation and a proof of Theorem 2}\label{sec_skein}
We already proved that $U_L$ depends only on the link $L$. In order to identify $U_L$ with the Alexander-Conway invariant, we prove that $U_L$ satisfies the skein relation $U_{L_+}-U_{L_-}=(\tm-\tp)U_{L_0}$ where the links $U_{L_+},U_{L_-},U_{L_0}$ are oriented and differ at only one place as in Fig.\ref{nin9}. 

Consider a given crossing in $L_+$ which we will resolve, we may assume that the strings at that crossing are oriented upwards. Then, pushing this crossing down and to the right we may assume that $L$ is presented as the plait closure of a braid with the fragment in Fig.\ref{nin9} at the right bottom. Therefore, in order to axiomatically define a Conway-type invariant it is enough to prove
a skein relation only for such a crossing.

\begin{theo}
The relation $U_{L_+} -
U_{L_-} = \td U_{L_0}$, holds, i.e. $$e(\beta_+)U_{\beta_+} -
e(\beta_-)U_{\beta_-} = \td e(\beta_0)U_{\beta_0},$$ where
$\beta_+ = \s_{2i-1}\s_{2i}^{-1}\s_{2i-1}\beta, \beta_- = \s_{2i-1}\s_{2i}\s_{2i-1}\beta,  \beta_0 = \s_{2i-1}^2\beta$, $L_0,L_+,L_-$ are the plait closure of $\beta_0,\beta_+,\beta_1$ respectively, and $\beta$ is an arbitrary braid in $B_{n,n}$.
\end{theo}

\begin{figure}[h]
\begin{center}
\begin{tikzpicture}[scale=0.7,>=latex]
\draw (0,0)--++(0,1)--++(0.3,0.3);
\lin{0}{0}{1}{-0.5}{0.2}{<}
\lin{2}{0}{3}{-0.5}{0.2}{<}

\draw [decoration={markings, mark=at position 0.5 with {\arrow{>}}},postaction={decorate}] (2,4.5)--(2,5);
\draw [decoration={markings, mark=at position 0.5 with {\arrow{>}}},postaction={decorate}] (0,4.5)--(0,5);

\draw (1,0)--++(0,1)--++(-1,1)--++(0,1)--++(0.3,0.3);
\draw (0.7,1.7)--++(1.3,1.3)--++(0,2);
\draw (2,0)--++(0,2)--++(-0.3,0.3);
\draw (3,0)--++(0,5);
\draw (1.3,2.7)--++(-1.3,1.3)--++(0,1);
\draw (0.7,3.7)--++(0.3,0.3)--++(0,1);
\draw[dotted, thick] (1.5,2.5) circle (1.2); 

\begin{scope}[xshift=4cm]
\draw (0,0)--++(0,1)--++(0.3,0.3);
\lin{0}{0}{1}{-0.5}{0.2}{<}
\lin{2}{0}{3}{-0.5}{0.2}{<}

\draw [decoration={markings, mark=at position 0.5 with {\arrow{>}}},postaction={decorate}] (2,4.5)--(2,5);
\draw [decoration={markings, mark=at position 0.5 with {\arrow{>}}},postaction={decorate}] (0,4.5)--(0,5);

\draw (1,0)--++(0,1)--++(-1,1)--++(0,1)--++(0.3,0.3);
\draw (0.7,1.7)--++(0.6,0.6);
\draw (1.7,2.7)--++(0.3,0.3)--++(0,2);
\draw (2,0)--++(0,2)--++(-2,2);
\draw (3,0)--++(0,5);
\draw (1.3,2.7)--++(-1.3,1.3)--++(0,1);
\draw (0.7,3.7)--++(0.3,0.3)--++(0,1);
\draw[dotted, thick] (1.5,2.5) circle (1.2); 
\end{scope}

\begin{scope}[xshift=8cm]
\draw (0,0)--++(0,1)--++(0.3,0.3);
\lin{0}{0}{1}{-0.5}{0.2}{<}
\lin{2}{0}{3}{-0.5}{0.2}{<}

\draw (1,0)--++(0,1)--++(-1,1)--++(0,1)--++(0.3,0.3);
\draw (0.7,1.7)--++(0.3,0.3)--++(0,1);
\draw (2,0)--++(0,5);
\draw  (3,0)--++(0,5);
\draw [decoration={markings, mark=at position 0.5 with {\arrow{>}}},postaction={decorate}] (2,4.5)--(2,5);
\draw [decoration={markings, mark=at position 0.5 with {\arrow{>}}},postaction={decorate}] (0,4.5)--(0,5);

\draw (1,3)--++(-1,1)--++(0,1);
\draw (0.7,3.7)--++(0.3,0.3)--++(0,1);
\draw[dotted, thick] (1.5,2.5) circle (1.2); 
\end{scope}
\end{tikzpicture}
\caption{Skein relation}\label{nin9}
\end{center}
\end{figure}
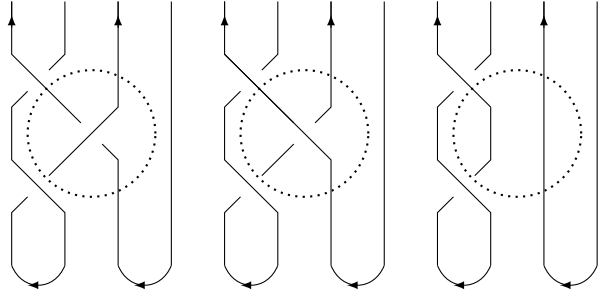

In Fig.\ref{nin9} the bottom part of a braid is depicted: a crossing and
its resolutions; these two strings are of the same color therefore the
resolution preserves their orientation. 

\begin{proof} Let us calculate the action.

$$
\s_{2i-1}\s_{2i}^{-1}\s_{2i-1} \left\{
\begin{array}{lcl}
 d_{i-1} \to d_i + d_{i-1} - \tp s_i + (t-1)(d_i + \td s_i) \\
 s_i \to \tp d_i - (t-1)s_i \\
 d_i \to \tp s_i - (t-1)(d_i + \td s_i) \\
 s_{i+1} \to s_i + s_{i+1} - \tp d_i + (t-1) s_i \\
\end{array}
\right.
$$

$$
\s_{2i-1} \s_{2i}\s_{2i-1} \left\{
 \begin{array}{lcl}
   d_{i-1} \to d_{i-1} + d_i - \tm s_i \\
   s_i \to \tm d_i \\
   d_i \to \tm s_i \\
   s_{i+1} \to s_i + s_{i+1} - \tm d_i \\
  \end{array}
\right.
$$

$$
\s_{2i-1}^2 \left\{
 \begin{array}{lcl}
   d_{i-1} \to d_{i-1} + (\tp - \tm)s_i \\
   s_i \to s_i \\
   d_i \to d_i - (\tp - \tm)s_i  \\
 \end{array}
\right.
$$

\begin{rem}
\label{simple}
Let $\beta, \eta $ be two arbitrary braids in $B_{n,n}$,
hence $$U_{\eta\beta} = \sum_{\sigma\in S_{n-1}}
\varepsilon(\sigma)\prod_{i=1}^{n-1}r_i\times \beta (\eta
s_{\sigma(i)})$$ 
\end{rem}

According to this remark, in order to compute
$U_{\beta_+}$, we should replace $\ldots
\wedge [\beta s_i]\wedge [\beta s_{i+1}]\ldots $ with $\ldots \wedge
[\beta(\tp d_i - (t-1)s_i)]\wedge [\beta(s_i + s_{i+1} - \tp d_i +
(t-1) s_i)]\wedge\ldots$ in the definition of $U_{\beta}$. We are
interested only in the right side of definition, i.e. in $[\beta s_1]\wedge \dots \wedge
[\beta s_{n-1}]$. Since $\beta$ is a linear map, we can write the difference between $\beta(\ldots \wedge [s_i]\wedge
[s_{i+1}]\wedge\ldots)$ and $$\beta(\ldots \wedge[(\tp d_i -
(t-1)s_i)]\wedge [(s_i + s_{i+1} - \tp d_i + (t-1) s_i)]\wedge\ldots)$$

Notice that  $e(\beta_+) = e(\beta)\tm, e(\beta_-) = e(\beta)\tp,
e(\beta_0) = e(\beta)$.

Therefore $e(\beta_+)U_{\beta_+} - e(\beta_-)U_{\beta_-} = \td
e(\beta_0)U_{\beta_0}$ becomes, after cancellation

$$[\tm(((1-t)s_i + \tp d_i)]\wedge [(s_{i+1} + ts_i - \tp d_i))] - \tp((\tm [d_i] \wedge [(s_i + s_{i+1} - \tm d_i)]))=$$
$$= \td [s_i]\wedge [s_{i+1}]$$

which can be verified by a direct calculation.

\end{proof}
Therefore, the skein relation identify $U_L$ with Alexander-Conway polynomial up to scaling. The scaling constant is $1$ because of Example \ref{ex1}. 
\section{Remarks}\label{sec_rem}
1. Higher Alexander polynomials (a.k.a.  elementary ideals of the Alexander matrix) can be obtained in a similar way. We look
at how $B'$ (see Section \ref{sec_aux}) changes during transformations of a braid preserving
its plait closure. One can verify that the ideals generated by the determinants
of all submatrices of given size are also invariant under braid
transformations. Direct calculations seem to be not very enlightening. But it is enough to prove that $B'$  is the presentation matrix 
of the Alexander module. See below a sketch:

a) Each homology class of the infinite cyclic
cover over complement of a knot can be pushed to the top of the braid,
therefore the circles $d_i\ (i=1,\dots,n-1)$ at the top of a braid can be
chosen as generators of the Alexander module.

b) The loops $s_i$, and, therefore  $\beta s_i$ are homologous to $0$ in the homologies of
cyclic covering. If we write $0=[\beta s_i] =\sum c_i[s_i]+\sum c_i'[d_i] $,
then we see that the row corresponding to $[s_i]$ in
$B'$  is a relation for $[d_j]$.

c) Now we prove that these are all the relations. Indeed, consider a linear combination $[d] = \sum c_j [d_j], c_j\in\mathbb Z[\tp,\tm]$, suppose
that $[d] = 0$ in the homologies of the cyclic covering. 
This implies that it is possible to find a surface $S$ with boundary $d$;
this surface {\bf does not} intersect $K$.
Stretching $S$ closer to the knot, we look at the Morse decomposition
of $S$. So, $S$ has a number of ``hats''
at the top (neighbourhoods of local maxima) and at the bottom
(neighbourhoods of local minima), each of those is the circle $s_i$ spanned by a disk. 
That means that $[d] =\sum_{i=1}^n (a_i[s_i]+ b_i\beta([s_i])$ in $H_1(D',\Theta)$ where $a_i,b_i\in\mathbb Z[\tp,\tm]$.

2. The construction of Bigelow \cite{Bigelow} is done in the language of coverings,
but it can be easily reinterpreted in terms of a local system on $D'$
where counterclockwise rotation around a marked point gives a multiplication by $t$
in the fiber. Then we lift this local system to symmetric power and
add the condition that a rotation
around the diagonal gives multiplication by $-t^{-1}$. Then Jones polynomial
is obtained via intersection of two cycles in this symmetric
power. Higher polynomials could be obtained as higher Alexander
polynomials, but the direct construction fails because of the above
monodromy around the diagonal.  



\begin{thebibliography}{99}

\bibitem[J.Alexander, 1928]{Alexander} J. W. Alexander, Topological Invariants of knots and links Trans.Amer.Math.Soc {\bf
30} 275-306, 1928.

\bibitem[J.Birman, 1976]{Birman} J. S. Birman, On the stable equivalence
of plat representations of knots and links, Canad. J. Math. {\bf
28}, no. 2., 264-290, 1976.

\bibitem[S.Bigelow, 2001]{Bigelow} S. J. Bigelow, A homological
definition of the Jones polynomial, Geometry and Topology Monographs
Volume 4: Invariants of knots and 3-manifolds (Kyoto) Pages
29-41, 2001.

\bibitem[S.Bigelow, A.Cattabriga, V.Florens, 2012]{vincent} S. J. Bigelow, A. Cattabriga, V. Florens
  Alexander representation of tangles, arXiv:1203.4590.

\bibitem[J.Conway, 1967]{Conway} J. H. Conway, An enumeration of knots and links and some of their algebraic properties,
In: Computational Problems in Abstract Algebra, Proc. Conf. Oxford
(edited by J. Leech), pp. 329-358; New York: Pergamon Press.
MR 41:2661, 1967.

\bibitem[R.Crowell,R.Fox, 1963]{CrowFox} R. H. Crowell and R. H. Fox, Introduction to Knot Theory, New York: Ginn and Co. (1963),
or: Grad. Texts Math. 57, Berlin-Heidelberg-New York: Springer
Verlag (1977). MR 26:4348; MR 56:3829.

\bibitem[R.Fintushel, R.Stern, 1996]{FintushelStern} R. Fintushel, R. J. Stern, Knots, Links, and
4-Manifolds,  Inventiones mathematicae, 134(2), 363-400, 1996.

\bibitem[A.Floer, 1988]{Floer} A. Floer. Morse theory for Lagrangian intersections. J. Differential Geometry, 28:513-547, 1988.

\bibitem[R.Fox, 1961]{Fox61} R. H. Fox., A quick trip through knot theory, In Topology of Three Manifolds - Proceedings
of 1961 Topology Institute at Univ. of Georgia, edited by M. K.
Fort, pp. 120-167. Englewood Cliffs, N. J. : Prentice-Hall. MR {\bf
25}:3522.

\bibitem[R.Fox,J.Milnor, 1966]{FoxMilnor} R. H. Fox and J. Milnor, Singularities of 2- spheres in 4 - space and cobordism of knots.
Osaka J. Math., Vol. 3, pp. 257-267. MR 35:2273, 1966.

\bibitem[M.Freedman, F.Quinn, 1990]{FreedmanQuinn} M. H. Freedman and F. Quinn, Topology of 4-manifolds,
Princeton Mathematical Series, vol 39, Princeton University Press,
Princeton, NJ, 1990.

\bibitem[L.Kauffman, 1983]{Kauffman83} L. H. Kauffman, Formal Knot Theory, Mathematical Notes No. 30, Princeton University
Press. MR 85b:57006, 1983.

\bibitem[L.Kauffman, 2001]{Kauffman} L. H. Kauffman, Knots and Physics (Series on Knots and Everything, Vol. 1), World Scientific Publishing
Company, 2001.

\bibitem[M.Khovanov, 2006]{Khovanov} M. Khovanov, Link homology and categorification, Proceedings of the ICM-2006, Madrid, vol.2 989--999.

\bibitem[R.Lawrence, 1993]{Lawrence} R. J. Lawrence, A functorial
  approach to the one-variable Jones polynomial, Journal of
  Differential Geometry, 37, 689-710, 1993.

\bibitem[W.Lickorish, 1997]{Lickorish} W. B. Raymond Lickorish, An Introduction
to knot theory. Graduate Texts in Mathematics, 175. Springer-Verlag,
New York, 1997.

\bibitem[S.Manfredini, 1997]{generate} S. Manfredini Some subgroups of Artin's
  braid group. Topology and its Applications 78.1: 123-142, 1997.

 
\bibitem[P.Ozsv\'ath, Z.Szab\'o, 2004]{OzsvathSzabo} P. Ozsv\'ath,
  Z. Szab\'o, Holomorphic disks and knot invariants. Adv. Math. 186
  (1), 58-116, 2004.

\bibitem[J.Rasmussen, 2003]{Rasmussen} J. Rasmussen, Floer homology
  and knot complements. PhD thesis, Harvard University, 2003.

\bibitem[H.Seifert, 1934]{Seifert} H. Seifert, \"Uber das Geschlecht
  von Knoten, Math. Ann., Vol. 110, pp. 571-592, 1934.

\end{thebibliography}
\end{document}